\documentclass{article}

\usepackage{graphicx}
\usepackage{amsfonts} \usepackage{amsmath, amsthm}
\usepackage{epsfig} \usepackage{verbatim}

\newtheorem{Theorem}{Theorem}
\newtheorem{Proposition}{Proposition}
\newtheorem{Lemma}{Lemma}

\newtheorem{Remark}{Remark}

\begin{document}
\title{Lorentz Process with shrinking holes in a wall}
\author{P\'{e}ter N\'{a}ndori, Domokos Sz\'{a}sz
%\thanks{
%The support of the Hungarian National Foundation for Scientific
%Research grant No. K 71693 is gratefully acknowledged.
%The authors also thank the kind hospitality of the Fields
%Institute (Toronto) where - during June 2011 -  part of this work was done.
%}
}
\maketitle

\begin{abstract}
We ascertain the diffusively scaled limit of a periodic Lorentz
process in a strip with an almost reflecting wall at the origin.
Here, almost reflecting means that the wall contains
a small hole waning in time. The
limiting process is a quasi-reflected Brownian motion, which is
Markovian but not strong Markovian. Local time results for the periodic Lorentz process, having  independent interest, are also found and used. 
\end{abstract}

\textbf{
The periodic Lorentz process is a fascinating
non-linear, chaotic model that has been deeply investigated in the
last decades. The model is very simple: a massless point particle
moves freely in the plane (or in our case, in a strip) until
it hits one of the periodically situated smooth convex scatterers,
when it is reflected. The limit of the diffusively scaled trajectory
of the particle is known to be the Brownian motion. Further,
if the particle is restricted to a half strip, then the scaling
limit is going to be the so-called reflected Brownian motion.
Here we
introduce a {\it time-dependent scatterer configuration} (by adding a vertical wall with a shrinking hole) that
almost confines the particle to the half strip in such a way
that the scaling limit is a {\it quasi-reflected Brownian motion}, a natural generalization of both the
Brownian motion and the reflected Brownian motion.
}

\section{Introduction}  Lorentz models were introduced by H. Lorentz in 1905 for understanding the motion of a classical electron in a crystal (cf. \cite{H}).
In the last decade after a broad and thorough study of
Sinai billiards - or equivalently of periodic Lorentz processes - the non-homogeneous case got also widely examined. Here, {\it non-homogeneity} may appear either
in time (cf. \cite{Ch-D} as to  a mechanical model of Brownian motion or   \cite{GR-KST},  \cite{DK}, \cite{LChK} as to models of Fermi acceleration) or in space (cf. \cite{D-Sz-V} as to local perturbations of periodic Lorentz processes).
In the present work we investigate a question with non-homogeneity in time. Consider a periodic
Lorentz process with a finite horizon given in a horizontal strip, where the scatterer
configuration is assumed to be symmetric with respect to a vertical axis - through the origin, say.
Now, put a vertical wall at the symmetry axis and a tiny hole onto
the wall. The hole is getting smaller and smaller with time, thus giving
the particle less and less chance to cross the wall. It is an intriguing question at which speed the hole should shrink to result
a non trivial scaling limit of the trajectory of the particle (if such a speed exists at all).
Here, non trivial means that it is neither Brownian motion (BM), nor reflected
Brownian motion (RBM) since, if the hole was of full size or absent, then
these two processes would appear in the limit (see \cite{D-Sz-V}).\\

Indeed, if one takes the hole arbitrarily small, but fixed of size $\varepsilon > 0$,
then the limiting process is a BM whereas if the hole is empty,
then it is a RBM. The essence of this observation is that the limiting process does not change continuously as $\varepsilon \to 0$ and our goal is precisely to understand the situation when the limit is taken in a more delicate, time-dependent way.\\

To be more precise, let the configuration space in the absence of the wall be $\mathcal{D}:= (\mathbb{R} \times [0,1])
\setminus \cup_{i=1}^{\infty} O_i$. Here, $\{ O_i \}_i$ is a $\mathbb Z$-periodic extension of
a finite scatterer configuration in the unit square, which consists of
strictly convex, pairwise disjoint scatterers, with ${C}^3$
smooth boundaries, whose curvatures are bounded from below by a positive constant.
Further, we assume that $\cup_{i=1}^{\infty} O_i$ is
symmetric with respect to
the $y$-axis. The wall without the hole is $W_\infty = \{(x, y) \in \mathcal D |\, x=0\} = \cup_{k=1}^K [\mathcal{J}_{k,l}, \mathcal{J}_{k,r}]$
where the subintervals of the $y$-axis, denoted by $[\mathcal{J}_{k,l}, \mathcal{J}_{k,r}]$,
are the connected components of $W_\infty$. For later reference, put
\[ c_1 = \sum_{k=1}^{K} (\mathcal{J}_{k,r} - \mathcal{J}_{k,l}).\]
The {\it holes} will be subintervals $I_n \subset W_\infty$, thus we will be
considering a sequence $\{W_n = W_\infty \setminus I_n\}_n$ of walls.
Now, the $n$-th configuration space of the {\it billiard flow} is
$\mathcal{D}_n:= (\mathbb{R} \times [0,1])
\setminus (W_n \cup \cup_{i=1}^{\infty} O_i)$.
A massless point particle moves inside $\mathcal{D}_n$ (at time $t=0$ the first hole is present,
i.e. $n=1$) with unit speed until
it hits the boundary $\partial \mathcal{D}_n$. Then it is reflected by the classical
laws of mechanics (the angle of incidence equals to the angle of reflection)
and continues free movement (or free flight) in $\mathcal{D}_{n+1}$. Thus, at the time
instant of each reflection, the hole is replaced by an other one (meaning that the shrinking rate of the hole corresponds to real time). We also mention
that the reflections on the horizontal boundaries of the strip does not play any role
in our study. Thus one could define the horizontal direction to be periodic
(formally replace $[0,1]$ by $S^1$ in the definitions of $\mathcal D$ and
$\mathcal{D}_n$) yielding the same results (with some different limiting variance).\\
Since we change the configuration space in the moment of the reflection,
it is more convenient to use the discretized version of the billiard flow
(the usual Poincar\'e section, which is often called billiard ball map).
Thus define the {\it phase spaces}
\[ \mathcal{M}_n =\{ x = (q,v), q \in \partial \mathcal{D}_n, v \in S^1,
\langle v,u \rangle \geq 0 \text{ if $q \in \partial \mathcal D$  }\},\]
where $u$ denotes the inward unit normal vector to $\partial \mathcal D$
at the point $q \in \partial \mathcal D$. Here, $q$ denotes the position
of the particle at a collision and $v$ is the postcollisional velocity
vector.
If $q \in \partial \mathcal D$, $v$ can be naturally parametrized
by the angle between $u$ and $v$ which is in the interval
$[-\pi /2, \pi/2]$. If $q \in \partial W_n = W_n$, one can parametrize
$v$ by its angle to the horizontal axis. Thus, if this angle is in the interval
$[-\pi /2, \pi/2]$, then the particle is on the right-hand side of the wall, while
it is on the left-hand side if this angle is either in the interval
$[\pi /2, \pi]$ or in $(- \pi, - \pi/2]$.\\
Thus, the discretized version of the previously described billiard flow can be
defined by the {\it billiard ball maps} $\mathcal{F}_n :
\mathcal{M}_n \rightarrow \mathcal{M}_{n+1}$. Further, denote by
$\kappa_n: \mathcal{M}_n \rightarrow \mathbb R$ the projection to the horizontal direction of the {\it free flight} vector
from  $\mathcal{M}_n$ to $\mathcal{M}_{n+1}$ (that is, if $x=(q,v) \in
\mathcal{M}_n$ and $\mathcal{F}_n(x) =
(\tilde q, \tilde v)$, then $\kappa_n (x)$ is the projection to the horizontal axis
of the vector $\tilde q - q$). We also assume that the billiard has
{\it finite horizon}, meaning that, in the $\mathbb{Z}^2$-periodic extension
of the scatterer configuration, there is no infinite line on the plane
that would be disjoint to all the scatterers. Further,
write $\mathcal{I}_n = \{ I_k\}_{1 \leq k \leq n }$ for the collection of
the first $n$ holes, and
\[ S_n(x,\mathcal{I}_n)=S_n (x) = \sum_{k=1}^n
\kappa_k \mathcal{F}_{k-1} \dots \mathcal{F}_1(x),\]
where $x \in \mathcal{M}_1$. \\
%and for fix $x$ and $\mathcal{I}_n$, define
%$S_t(x,\mathcal{I}_n)=S_t (x)$ for $t\geq 0$ as the piecewise linear,
%continuous extension of $S_n(x)$.\\
What remains is the definition of the holes $I_n$.
%(briefly, write $\mathcal{I}_n = \{ I_k\}_{1 \leq k \leq n }$).
For this,
 fix some sequence $\underline{\alpha} = (\alpha_n)_{n \geq 1}$ with $\alpha_n \rightarrow 0$ and, independently
 of each other, choose uniformly
 distributed points $\xi_n,$ $n \geq 1$ on
 $\cup_{i=1}^{K} [\mathcal{J}_{i,l}, \mathcal{J}_{i,r}]$.
% (if $\mathcal J_{i, r} - \alpha_n < J_{i, l}$,
% then $[\mathcal{J}_{i,l}, \mathcal{J}_{i,r} - \alpha_n] = \emptyset$ by definition).
We will use the following three special choices:
\begin{enumerate}
\item Assume that $\xi_n \in [\mathcal{J}_{i,l}, \mathcal{J}_{i,r}]$,
and denote $l_n = \mathcal{J}_{i,r} - \xi_n$. If $l_n > \alpha_n$, then put
 $I_n = (\xi_n, \xi_n + \alpha_n ) $, otherwise put
 $I_n = (\xi_n, \mathcal{J}_{i,r} ) \cup
 (\mathcal{J}_{i,l}, \mathcal{J}_{i,l} + \alpha_n - l_n)$, which is a subset
 of $W_{\infty}$ for
$n$ large enough. With this particular choice, write
\[ S^{\searrow}_n(x, \underline{\alpha}) =  S^{\searrow}_n(x) =
 S_n(x,\mathcal{I}_n) \]
 and
 \[ \mathcal{F}^{\searrow}_n = \mathcal{F}_n.\]
\item For each $1\leq k \leq n $, let the random variables $\xi_n^{(k)}$
be independent and distributed like $\xi_n$.
Assume that $\xi_n^{(k)} \in [\mathcal{J}_{i,l}, \mathcal{J}_{i,r}]$,
and denote $l_{n}^{(k)} = \mathcal{J}_{i,r} - \xi_n^{(k)}$.
If $l_{n}^{(k)} > \alpha_n$, then put
 $I_n^{(k)} = (\xi_n^{(k)}, \xi_n^{(k)} + \alpha_n ) $, otherwise put
  $I_n^{(k)} = (\xi_n^{(k)}, \mathcal{J}_{i,r} ) \cup
   (\mathcal{J}_{i,l}, \mathcal{J}_{i,l} + \alpha_n - l_n^{(k)})$, and finally
$\mathcal{I}_{n} = (I_n^{(k)})_{1 \leq k \leq n }$.
%Define
%$\mathcal{I}_{n} =  \{ (\xi_n^{(k)}, \xi_n^{(k)} + \alpha_n ) \}_{1 \leq k \leq n }$.
With this particular choice, write
\[ S^{\equiv}_n(x, \underline{\alpha}) =  S^{\equiv}_n(x) =
 S_n(x,\mathcal{I}_{n}). \]
\item Let $I_n = W_{\infty}$. With this particular choice, write
\[ S^{(per)}_n( x) =
 S_n(x,\mathcal{I}_n), \]
and for a fixed $x$, define
$S_t^{(per)}(x)$ for $t\geq 0$ as the piecewise linear,
continuous extension of $S_n^{(per)}(x)$. Finally, write
\begin{eqnarray*}
\mathcal{F}^{(per)} &=& \mathcal{F}_1,\\
\mathcal{M}^{(per)} &=& \mathcal{M}_1.
\end{eqnarray*}
\end{enumerate}

Here the first choice - the only really time dependent - is the most interesting one.
In the second case, one has to redefine the whole trajectory segment
$S_1^{\equiv}, \dots S_n^{\equiv}$ for each $n$, thus we have a
sequence of billiards (in other words, the increments of $S_n^{\equiv}$
form a double array), while the third one is just a usual
periodic Lorentz process.\\
There is a natural measure - the projection of the {\it Liouville measure}
of the periodic billiard flow - on $\mathcal{M}^{(per)}$ which is invariant
under $\mathcal{F}^{(per)}$. Denote the restriction of this measure
to the two neighboring tori to the origin by $\bf P$. Note that $\bf P$ is
finite, so normalize it to be a probability measure.\\
Finally, define $\mathcal{J} \subset \mathcal{M}^{(per)}$
as such points on the discrete phase space without any wall,
from which before the forthcoming collision, the particle crosses
 $\cup_{i=1}^K (\mathcal{J}_{i,l}, \mathcal{J}_{i,r})$.
 Note that the finite horizon condition implies that $\mathcal{J}$
 is bounded.\\

%$\mathfrak{S} \mathfrak{B} \mathfrak{Q} \mathfrak{W} \mathfrak{S} \mathfrak{L}
%\mathcal{\mu} \mathcal{B}$
Now we proceed to the definition of the limiting processes. (The intuition behind their appearance in our result and in its proof as well will be explained after the formulation of the theorem.) Since we are
going to have two very similar processes, we call both quasi-reflected
Brownian motions and distinguish between them only in the abbreviation.\\
Consider a BM $\mathfrak{B} = (\mathfrak{B}_t)_{ t \in [0,1]}$
with parameter $\sigma$ on $[0,1]$.
Its local time at the origin is denoted by $\mathfrak{L} = (\mathfrak{L}_t)_{
t \in[0,1]}$. That is,
\[ \mathfrak{L}_t = \lim_{\varepsilon \searrow 0} \frac{1}{2 \varepsilon}
\int_{0}^{t} 1_{\{ |\mathfrak{B}_S| < \varepsilon \}} ds.\]
Now, given $\mathfrak B$, consider a  Poisson Point
Process $\Pi$ with intensity measure $c d \mathfrak{L}$ with some positive constant $c$. The intuition behind this process is roughly speaking the following: since the local time describes the relative time process a Brownian motion spends in an infinitesimal neighborhood of a point, in our case of the origin, it can also be interpreted as telling us the density process of number of visits of the origin by the Lorentz process. Out of them only those visits are successful, i. e. resulting in getting to the other side of the wall, when the particle hits the hole, and these instants of time are precisely given by a Poisson process -- according to the Poisson limit law.
With probability
one, the support of the measure $c (d \mathfrak{L})$ is $\mathfrak Z$,
where $\mathfrak Z = \{ s: 0 \leq s \leq 1:
\mathfrak{B}_s =0 \} $ is the zero set of $\mathfrak{B}$.
Denote the points of $\Pi$ by $P_1, P_2, ...$ in decreasing order.
In fact, $\Pi$ has finitely many points. If it has $m$ points, then put
$P_{m+1} = P_{m+2}=...=0$. Further, write $P_0 = 1$ and introduce a
Bernoulli distributed random variable $\eta$ with parameter $1/2$
(where the parameter means the probability of being equal to 1)
 which is independent
  of $\mathfrak{B}$ and $\Pi$.\\
Now, the process $\mathfrak{Q} = (\mathfrak{Q}_t)_{ t \in [0,1]}$ with
$\mathfrak{Q}_0=0$ and
$$
\mathfrak{Q}_t = \left\{ \begin{array}{rl}
  (-1)^{\eta} |\mathfrak{B}_t| &\mbox{ if $\exists n \in \mathbb{Z}_{+} \cup \{0\}:
    t \in ( P_{2n+1},P_{2n}] $} \\
       (-1)^{1- \eta} |\mathfrak{B}_t| &\mbox{ otherwise}
                      \end{array} \right.
$$
is called the quasi-reflected Brownian motion with parameters $c$ and $\sigma$,
and denoted
by qRBM(c,$\sigma$). \\
%($\eta$ reflects from which side of the origin our Lorentz process has started.)\\
The definition of QRBM is
similar to that of qRBM. The difference is that
 $c (d \mathfrak{L})$ now should be replaced by $c \frac{1}{\sqrt t} (d\mathfrak{L}_t)$. As a result, the
 Poisson process will have infinitely many points, which accumulate only at the origin.
 Now, denote by $P_1, P_2, \dots$ these points in decreasing order (N. B.:
 there is no smallest one among them), put $P_0=1$ and define $\eta$
 and QRBM(c,$\sigma$) as before.
\begin{Remark}
\label{remark1}
One can easily check the following statements.
The qRBM(c,$\sigma$) is almost surely continuous on $[0,1]$,
homogeneous Markovian but not strong Markovian
(think of the stopping time
$T = \min  \{ t>1/2 : \mathfrak{Q}_t = 0 \} \wedge 1$)
and $\mathfrak{Q}_t$ has Gaussian distribution with mean zero and variance $t \sigma^2$.\\
The QRBM, similarly to the qRBM, is continuous,
Markovian (however not time homogeneous), not strong Markovian,
and has the same one dimensional distributions as qRBM.
Contrary to the qRBM, the QRBM is self similar in the
following sense: if $\mathfrak{Q}_t$ is a QRBM, then
\[  (\mathfrak{Q}_t)_{t \in[0,1/p]} \overset{d}{=}
\left( \frac 1{\sqrt p} \mathfrak{Q}_{pt} \right)_{t \in[0,1/p]}, \]
where $1<p$.\\
Further, one can easily extend the definition of both
processes to $\mathbb{R}_{+}$.

\end{Remark}

As usual, $C[0,1]$ will denote the space of continuous functions and $D[0,1]$ the Skorokhod space over $[0, 1]$ (for the definition of the latter, we
refer to \cite{Billingsley}).
We will also use evident modifications, for instance, $D_{\mathbb{R}^2} [t_0,1]$ will denote
the Skorokhod space of $\mathbb{R}^2$-valued functions over an interval $[t_0, 1]$.\\
%Define the function ${\bf W}^{\searrow}_n : [0,1] \rightarrow \mathbb R$
%as ${\bf W}^{\searrow}_n (t) = S^{\searrow}_{nt}/ \sqrt n$.
%Let $\mu^{\searrow}_n$ denote the measure on
% $C[0,1]$ induced by ${\bf W}^{\searrow}_n$, where the initial distribution, i.e. the distribution
%of $x$ is given by $\bf P$.
Let the function ${\bf W}^{\searrow}_n$ be the following:
${\bf W}^{\searrow}_n(k/n) =
 S^{\searrow}_k/ \sqrt n$
  for $0 \leq k \leq n$ and define ${\bf W}^{\searrow}_n(t)$ for
  $t \in [0,1]$ as its piecewise linear, continuous extension.
Let $\mu^{\searrow}_n$ denote the measure on
 $C[0,1]$ induced by ${\bf W}^{\searrow}_n$, where
 the initial distribution, i.e. the distribution
 of $x$, is given by $\bf P$.
% the Liouville measure on the phase space of
% the billiard map on $\mathcal{T}$ restricted to the two neighboring tori
% to the origin.
 Analogously, define $\mu^{\equiv}_n$
 with ${\bf W}^{\equiv}_n$, where ${\bf W}^{\equiv}_n(k/n) =
  S^{\equiv}_{k}/ \sqrt n$.\\
Now, we can formulate our main result.

\begin{Theorem}
\label{tetel1}

There are positive constants $\sigma$ and $c_2$ depending only on the
periodic scatterer configuration, such that

 \begin{enumerate}

 \item if $\exists c>0: \alpha_n \sqrt n \rightarrow c$, then $\mu^{\searrow}_n$
 converges weakly to the measure induced by $QRBM(c_2 c, \sigma)$.
 \item if $\exists c>0: \alpha_n \sqrt n \rightarrow c$, then $\mu^{\equiv}_n$
   converges weakly to the measure induced by $qRBM(c_2 c, \sigma)$.
 \item if $\alpha_n \sqrt n \rightarrow 0$, then both $\mu^{\searrow}_n$
 and $\mu^{\equiv}_n$
 converge weakly to the convex combination of the measures induced by RBM
 and -RBM with weights $1/2$.
 \item if $\alpha_n \sqrt n \rightarrow \infty$, then both $\mu^{\searrow}_n$
  and $\mu^{\equiv}_n$
 converge weakly to the Wiener measure.
 \end{enumerate}

\end{Theorem}
	
Returning to the intuitive picture provided at the introduction of the process $qRBM(c_2 c, \sigma)$, it, indeed, explains statement 2 of the theorem. Since, in the setup of the definition $\mu^{\searrow}_n$, the holes are not uniformly small, but are only decreasing as of order $\frac{1}{\sqrt n}$, the chances to get over the wall are larger but also decreasing as in the definition of $QRBM(c_2 c, \sigma)$.

Instead of introducing the holes on the wall one could think about
the wall as a {\it trapdoor}, i.e. sometimes it is open and then the particle
crosses it without collisions, other times it is closed. If one opens
the door randomly with probability $\alpha_n /c_1$, then obtains the same
result.

The analogue of Theorem \ref{tetel1} for random walks is,
of course, easy to formulate in the following way.
 Define the stochastic process $\mathfrak{S}_n$ by:
 $Prob (\mathfrak{S}_0 = 1) = Prob
 (\mathfrak{S}_0 = -1) = 1/2$
 and for $k > 0$:
 \[ Prob (\mathfrak{S}_{k+1} = \mathfrak{S}_{k} + 1 |
 \mathfrak{S}_{k} \neq 0  )=
 Prob (\mathfrak{S}_{k+1} =
 \mathfrak{S}_{k} - 1 | \mathfrak{S}_{k}
 \neq 0  ) = 1/2, \]
 and
 \begin{equation}
 Prob (\mathfrak{S}_{k+1} = \mathfrak{S}_{k-1} |
 \mathfrak{S}_{k} = 0)= 1 - \epsilon, \label{egy}
 \end{equation}
 \begin{equation}
 Prob (\mathfrak{S}_{k+1} = - \mathfrak{S}_{k-1} |
 \mathfrak{S}_{k} = 0 )=  \epsilon. \label{ketto}
 \end{equation}

Here - and also in the sequel - $Prob$ stands for some abstract probability measure.

In the definition of $\mathfrak{S}_k$ put first $\epsilon = \alpha_k$ and
denote by $\nu^{\searrow}_n$ the measure on
$C[0,1]$ induced by ${\bf W}_n$, where ${\bf W}_n(k/n)
=  \mathfrak{S}_k/ \sqrt n$ for $0 \leq k
\leq n$ and is linearly interpolated in between. Analogously, define
$\nu^{\equiv}_n$ for each $n$ with the choice $\epsilon = \alpha_n$.
Then, if we replace each $\mu$ with $\nu$ in Theorem \ref{tetel1}, then the statement
remains true (with $\sigma = c_2 =1$),
and can be proven the same way as we prove Theorem \ref{tetel1}.\\

In the next section, we discuss some results concerning the periodic
Lorentz process, that are necessary for proving Theorem \ref{tetel1}.
Finally, Section 3 contains the actual proof of Theorem \ref{tetel1}.

\section{Limit theorems for the periodic Lorentz Process}

In this section, we present some facts about the periodic Lorentz process
in a strip. Whereas Proposition 1 is simply a strengthening of Theorem 4.2 of \cite{Sz-V}, Proposition 3 is a completely new statement interesting in itself. For later reference, we need to introduce some
abstract stochastic processes.\\
As before, $\mathfrak B =(\mathfrak B_t)_{ t \in [0,1]}$ denotes
a BM with parameter $\sigma$ (to be specified later)
and $\mathfrak L =(\mathfrak L_t)_{ t \in [0,1]}$ is its local
time at the origin.
%We also write
%$\mathfrak L_H$ for the local time at the origin measured only
%during $H \subset [0,1]$. That is,
%\[ \mathfrak{L}_H = \lim_{\varepsilon \searrow 0} \frac{1}{2 \varepsilon}
%\int_{H} 1_{\{ |\mathfrak{B}_S| < \varepsilon \}} ds.\]
We also use the notation
$\mathfrak B^{a,t_0} =(\mathfrak B^{a,t_0}_t)_{ t \in [t_0,1]}$
for a BM  with
parameter $\sigma $ starting from $a$ at time $t_0$; and
$\mathfrak L^{a,t_0} =(\mathfrak L^{a,t_0}_t)_{ t \in [t_0,1]}$
denotes its local time at the origin.
%; and $\mathfrak L^{a,t_0}_H$ is
%the previous local time during $H \subset [t_0,1]$.
Finally,
$\mathfrak B^{a,t_0 \leadsto b, t_1} =
(\mathfrak B^{a,t_0 \leadsto b, t_1}_t)_{ t \in [t_0,t_1]}$
stands for a Brownian bridge with
parameter $\sigma $ starting from $a$ at time $t_0$ and arriving at
$b$ at time $t_1$ (that is heuristically a BM with pinned down endpoints),
and
$\mathfrak L^{a,t_0 \leadsto b, t_1} =
(\mathfrak L^{a,t_0 \leadsto b, t_1}_t)_{ t \in [t_0,t_1]}$
is the local time of $\mathfrak B^{a,t_0 \leadsto b, t_1}$
at the origin. For a thorough description of all these
processes, see \cite{RY}.\\
%, and
%$\mathfrak L^{a,t_0 \leadsto b, t_1}_H$ is
%the previous local time during $H \subset [t_0,t_1]$.\\
Similarly to the previous notations, denote by
$L_{nt},\ t \in [0,1]$ the number of
visits to $\mathcal{J}$ in the time interval
$[1, \lfloor nt \rfloor] $, and
$L_{H}$ is
the number of
visits to $\mathcal{J}$ in the time interval
$H$.\\

The first statement is a local limit theorem, formulated
in a fashion tailored to our purposes. For this, let $\phi$ denote
the density of the standard normal law. Now, the assertion reads
as follows.
%The most important ingredient of the proof of Theorem \ref{tetel1}
%is the following local limit theorem for a periodic Lorentz process in a strip.

\begin{Proposition}
\label{LLT}
Fix some positive integer $k$ and a subset
$\mathcal{Z}$ of the set $\{ 1, 2, \dots k\}$. For all
$i \in \{ 1, 2, \dots k\} \setminus \mathcal Z$,
let $t_i \in [0,1]$, $b^{(i)} \in \mathbb{R}$
be real numbers such that if
$i <j$ with $i,j \in \mathcal Z$, then $t_i < t_j$.
Write $b_n^{(i)} :=
\lfloor b^{(i)} \sqrt n\rfloor$ and $n_i = \lfloor n t_i \rfloor$ for
any positive integer $n$.
Define $n_0=b_n^{(0)}=0$. For
$i' \in \mathcal Z$, write $b_n^{(i')}=0$ and
choose some sequences $n_{i'}$ such that for any $i,j \in \{ 1, 2, \dots k\}$
with $i <j$, $n_i \leq n_j$ holds. Then
%and finally, for any
%$i \in \{ 1, 2, \dots k\} \setminus \mathcal Z$,
%$n_{i-1} + n \delta < n_i < n_{i+1} - n \delta$. Then
\begin{eqnarray*}
&{\bf P}& \big( \forall i \in \{ 1, 2, \dots k\} \setminus \mathcal Z:
\lfloor S^{(per)}_{n_i}(x) \rfloor
= b_{n}^{(i)};
\forall i' \in \mathcal Z:
\left( \mathcal{F}^{(per)}\right)^{n_{i'}}(x)
\in \mathcal{J} \big)\\
&=& c_0^{|\mathcal{Z}|} \prod_{i=1}^k \frac{\phi(
\frac{b_n^{(i)}-b_n^{(i-1)}}{\sigma \sqrt{n_i-n_{i-1}}})+ o_i(1)}
{\sigma \sqrt{n_i-n_{i-1}}},
\end{eqnarray*}
with some constants $\sigma$ and $c_0$ depending only
on the periodic scatterer configuration. Further, there exist
a sequence $\varsigma(n) \rightarrow 0$,
%\varsigma_{\delta,
%\{ b_i, i \in \in \{ 1, 2, \dots k\} \setminus \mathcal Z\}}(n)
%\rightarrow 0$ as $n \rightarrow \infty$,
%depending on $\delta$, $b^{(i)}$'s and $t_i$'s, but not depending
%on the choice of $n_i$'s for $i \in \mathcal Z$,
such that
$|o_i(1)| < \varsigma(n_i - n_{i-1})$ for all $i \in \{ 1, 2, \dots k\}$.
\end{Proposition}

Proposition \ref{LLT} is an extension of Theorem
4.2 in \cite{Sz-V} in two aspects. On the one hand, it is formulated
for $k$-tuples, while in \cite{Sz-V} it is only stated for
$k=1,2$. On the other hand, the error term is claimed to be uniform
in the choice of $n_i$ (it is, in fact, uniform in more general
choices of $b_n^{(i)}$, but we only use it for $b_n^{(i)}$ of
the form presented in Proposition \ref{LLT}). Both generalizations
follow from the proof presented in \cite{Sz-V}, thus we do not provide
a formal proof here. We also note that Proposition \ref{LLT} is an
extension of Proposition 3.6 in \cite{recprop}, too.
%We do not
%prove it here since it can be proven similarly as Theorem
%4.2 in \cite{Sz-V}.
From now on, all stochastic processes derived
from the BM will have parameter $\sigma$ of Proposition \ref{LLT}.
%(which is in fact given by the Green-Kubo formula).
\\
The next important fact is the weak invariance principle for the position,
which was first proven in \cite{BS} and \cite{BSC}.

\begin{Proposition}
\label{WIP}
\[\left( \frac{S^{(per)}_{nt}}{\sqrt n} \right)_{t \in [0,1]} \Rightarrow
(\mathfrak{B}_t)_{t \in [0,1]},\]
where $\Rightarrow$ stands for weak convergence in the space $C[0,1]$.
\end{Proposition}

The novelty of this section is in fact the following statement.
The position
of the particle and its local time at $\mathcal J$ jointly converge
to a BM and its local time at the origin (the latter being multiplied
by a constant). Formally,

\begin{Proposition}
\label{allitas1}
\[\left( \frac{S^{(per)}_{nt}}{\sqrt n}, \frac{L_{nt}}{\sqrt n} \right)_
{t \in [0,1]} \Rightarrow (\mathfrak{B}_t, c_0
\mathfrak{L}_t)_{t \in [0,1]}, \]
as $n \rightarrow \infty$ where the left hand side is understood
as a random variable with respect to the probability measure $\bf P$, and
$\Rightarrow$ stands for
weak convergence in the Skorokhod space $D_{\mathbb{R}^2} [0,1]$.
\end{Proposition}

\begin{proof}[Proof of Proposition \ref{allitas1}]

As usual, one has to check the convergence of finite
dimensional distributions and the tightness (see \cite{Billingsley}
for conditions implying weak convergence on some function spaces).\\
First, we prove the convergence of the finite dimensional distributions.
Note that the convergence of the first coordinate follows
from Proposition \ref{WIP} (and even from Proposition \ref{LLT}),
while the convergence of the second coordinate follows from an extended
version of the proof of Theorem 9 in \cite{recprop}. But the joint convergence
is a stronger statement then the convergence of the individual coordinates,
and it
requires a formal proof.\\
To obtain the joint convergence, first observe that the convergence of
the first coordinate (the rescaled position)
is well known, even in the local sense (eg. Proposition \ref{LLT}). Thus
we are going to prove that under the condition that the rescaled position is close
to some specific number, the second coordinate converges to the desired limit.
In order to do this computation, we need to define some new measures on $\mathcal{M}^{(per)}$.\\
 First,
 choose $0 < t_0 <1 $, $a\in \mathbb{R}$ and write
 $a_n = \lfloor \sqrt n a \rfloor $.
 Restrict the measure $\bf P$ to such points $x$ where
 $\lfloor S^{(per)}_{\lfloor nt_0 \rfloor}(x) \rfloor = a_n$
 and rescale it to obtain a
 probability measure. The resulting measure is denoted by $\bf P_n$.
 Thus, with the notation
 \[ {\mathcal A_1}= \mathcal A_1 (n) = \{ x: \lfloor S^{(per)}_{\lfloor nt_0 \rfloor}(x) \rfloor = a_n\}
 \subset \mathcal{M}^{(per)},\]
 for $M \subset \mathcal{M}^{(per)}$ measurable sets,
 ${ \bf P_n} (M) = { \bf P} (M \cap \mathcal A_1) / {\bf P} ( \mathcal A_1)$.
Then, choose $t_0 < t_1  <1 $, $b\in \mathbb{R}$ and write
 $b_n = \lfloor \sqrt n b \rfloor $.
Define
$\bf Q_n$ as the conditional measure of $\bf P_n$ on such points $x$,
where $\lfloor S^{(per)}_{\lfloor nt_1 \rfloor} (x)  \rfloor  =b_n$.
That is, with the notation
\[ \mathcal A_2 = \mathcal A_2(n) = \{ x: \lfloor S^{(per)}_{\lfloor nt_1 \rfloor}(x) \rfloor = b_n\}
\subset \mathcal{M}^{(per)},\]
for $M \subset \mathcal{M}^{(per)}$ measurable sets,
${\bf Q_n} (M) = {\bf P_n} (M \cap \mathcal A_2) / {\bf P_n} (\mathcal A_2)$.\\
Now, we prove the following lemma.

\begin{Lemma}
\label{slemma1}
%For every $0<\delta<(t_1-t_0)/2$,
\[ L_{ [n t_0, n t_1 ]} / (c_0 \sqrt n) \Rightarrow
\mathfrak{L}^{a,t_0 \leadsto b, t_1}_{t_1},\]
where $L_{ [nt_0, nt_1]} / (c_0 \sqrt n)$
is understood as a random variable with respect to $\bf Q_n$.
Similarly,
%For every $0<\delta<t_0$,
\[ L_{ nt_0 } / (c_0 \sqrt n) \Rightarrow
\mathfrak{L}^{0,0 \leadsto a, t_0}_{t_0},\]
where $L_{ t_0 } / (c_0 \sqrt n)$
is understood as a random variable with respect to $\bf P_n$.

\end{Lemma}

%\end{proof}

\begin{proof}[Proof of Lemma \ref{slemma1}]
%Denote by $\bf Q_n$ the conditional measure of $\bf P_n$ to such points $x$,
%where $\lfloor S^{(per)}_{\lfloor nt_1 \rfloor} \rfloor (x) =b_n$.
%Thus with the notation
%\[ A_2 = A_2(n) = \{ x: \lfloor S^{(per)}_{\lfloor nt_1 \rfloor}(x) \rfloor = b_n\}
%\subset \mathcal{M}^{(per)},\]
%for $M \subset \mathcal{M}^{(per)}$ measurable sets,
%${\bf Q_n} (M) = {\bf P_n} (M \cap A_2) / {\bf P_n} ( A_2)$.
%Proposition
%\ref{LLT} implies that the fraction appearing in (\ref{egydimkonv}) is
%asymptotically equal to ${\bf P_n} ( A_2)$. Thus,
%it is enough to prove that
%$L_{nt_1} / (c_0 \sqrt n)$ -
%as a random variable with respect to $\bf Q_n$ -
%weakly converges to the appropriate limit.\\
We prove only the first statement, since the second one can be proven
analogously. Similarly to the proof of Theorem 9 in \cite{recprop}, we are
going to use the method of moments (see \cite{Billingsley},
Chapter 1.7, Problem 4., for instance). That is,  we are going to estimate
\[ \mathbb{I}_n^{k} := \int ( L_{ [n t_0, n t_1]})^k d{\bf Q_n}. \]
For some fixed positive integer $k$ and for $\lfloor nt_0 \rfloor =n_0
< n_1< n_2< ...< n_k <
n_{k+1} = \lfloor nt_1 \rfloor$, define the set
\[ \mathcal A_3 = \mathcal{A}_3 (n_1, \dots, n_k)
= \{ x: \{ \left( \mathcal{F}^{(per)} \right)^{n_i}x,  1 \leq i \leq k \}
\subset \mathcal{J} \} \subset \mathcal{M}^{(per)}.\]
Representing $L_{[n t_0, n t_1 ]}$ as a sum of
$\lfloor n t_1 \rfloor - \lfloor n t_0 \rfloor + 1$ indicator variables,
one concludes
\begin{equation}
\label{sublemma_2012_1}
\mathbb{I}_n^{k} \sim k! \sum_{
\lfloor nt_0 \rfloor =n_0
< n_1< n_2< ...< n_k <
n_{k+1} = \lfloor nt_1 \rfloor }
{\bf Q_n} ( \mathcal A_3(n_1, \dots n_k)).
\end{equation}
In fact, there should be $k-1$ similar sums,
for $n_1 < \dots < n_l$, $1 \leq l \leq k-1$, respectively, but the contribution
of them is of smaller order of magnitude, as we will see in the forthcoming
computation.
Thus, we need to estimate ${\bf Q_n} (\mathcal A_3(n_1, \dots n_k))$.
By definition,
\begin{equation}
\label{sublemma_2012_2}
 {\bf Q_n} ( \mathcal{A}_3) = \frac{{\bf P}(\mathcal{A}_1 \cap
 \mathcal{A}_2 \cap \mathcal{A}_3 ) }{{\bf P}(\mathcal{A}_1 \cap
 \mathcal{A}_2 )}.
\end{equation}
Using Proposition \ref{LLT}, one obtains the asymptotic equalities
\begin{eqnarray}
&&\frac{{\bf P}(\mathcal{A}_1 \cap
 \mathcal{A}_2 \cap \mathcal{A}_3 )}{{\bf P} (\mathcal A_1)} \nonumber \\
 &\sim&
 \frac{c_0^k}{\sigma^{k+1} \left( 2 \pi \right)^{\frac{k-1}{2}}}
 \phi \left( \frac{a_n}{\sigma \sqrt{n_1-n_0}} \right)
 \phi \left( \frac{b_n}{\sigma\sqrt{n_{k+1}-n_k}} \right)
 \prod_{i=1}^{k+1} \frac{1}{\sqrt{n_i-n_{i-1}}}, \label{sublemma_2012_3}
\end{eqnarray}
and
\begin{equation}
\label{sublemma_2012_4}
\frac{{\bf P}(\mathcal{A}_1 \cap
 \mathcal{A}_2  )}{{\bf P} (\mathcal A_1)} \sim
 \frac {\phi \left( \frac{b_n - a_n}{\sigma \sqrt{n(t_1-t_0)}} \right) }
 {\sigma \sqrt{n (t_1-t_0)}}.
\end{equation}
Next, we substitute (\ref{sublemma_2012_2}) by the product of the
right hand sides of (\ref{sublemma_2012_3}) and (\ref{sublemma_2012_4})
in the sum of (\ref{sublemma_2012_1}). The resulting sum is a Riemann
sum which is asymptotically equal to the following Riemann integral
\begin{eqnarray}
&&
n^{\frac{k}{2}}
 c_0^k k! \sigma^{-k} (2 \pi)^{-\frac{k-1}{2}} \sqrt{t_1-t_0}
 \left[ \phi\left( \frac{b-a}{\sigma \sqrt{t_1-t_0}}\right)\right]^{-1}
 \label{3soros} \\
 &&{\int \dots \int}_{0 <s_1<s_2<...<s_k<t_1-t_0 } d \underline{s} \nonumber \\
 && \phi\left( \frac{a}{\sigma \sqrt{s_1}}\right)
 \frac{1}{\sqrt{s_1}} \frac{1}{\sqrt{s_2-s_1}} \dots
 \frac{1}{\sqrt{s_k-s_{k-1}}}
 \frac{1}{\sqrt{t_1-t_0-s_k}}
 \phi\left( \frac{b}{\sigma \sqrt{t_1-t_0-s_k}}\right), \nonumber
 \end{eqnarray}
 where $\underline{s}=(s_1, \dots, s_k)$.
Note that when we substituted (\ref{sublemma_2012_2}) by the product of the
right hand sides of (\ref{sublemma_2012_3}) and (\ref{sublemma_2012_4}),
we made an error. Due to Proposition \ref{LLT}, this error
is bounded by
\[C \sqrt n \varsigma(\min_{i}\{ n_i - n_{i-1}\})
\prod_{j=1}^{k+1} \frac{1}{\sqrt{n_j - n_{j-1}}},\]
with some constant
$C$.
Thus, in order to see that (\ref{3soros}) is asymptotically equal to
$\mathbb{I}_n^{k}$, it remains to prove that
\begin{equation}
\label{3soros_indoklasa}
\sum_{
\lfloor nt_0 \rfloor =n_0
< n_1< n_2< ...< n_k <
n_{k+1} = \lfloor nt_1 \rfloor }
\sqrt n \varsigma(\min_{i}\{ n_i - n_{i-1}\})
\prod_{j=1}^{k+1} \frac{1}{\sqrt{n_j - n_{j-1}}}
\end{equation}
is in $o \left( n^{\frac{k}{2}}\right)$.
To prove this, pick $\varepsilon>0$ small and $K$
such that $\varsigma(K) < \varepsilon$.
The sum over indices $n_1, \dots n_k$, where all $n_i - n_{i-1}$ is
larger then $K$, is asymptotically bounded by
\begin{eqnarray*}
&&\varepsilon n^{\frac{k}{2}}
{\int \dots \int}_{0 <s_1<s_2<...<s_k<t_1-t_0 } d \underline{s} \\
&&\frac{1}{\sqrt{s_1}} \frac{1}{\sqrt{s_2-s_1}} \dots
 \frac{1}{\sqrt{s_k-s_{k-1}}}
  \frac{1}{\sqrt{t_1-t_0-s_k}}.
\end{eqnarray*}
Now, choose a subset $H$ of the set $\{ 1, 2 \dots k+1\}$, with
$|H|=l \geq 1$. Then the sum over indices $n_1, \dots n_k$, where
$n_i - n_{i-1} \leq K$ for $i \in H$, and $n_i - n_{i-1} >K$ otherwise,
is asymptotically bounded by $K^l n^{\frac{k-l}{2}}$ multiplied by an integral
similar to the previous one. Thus, we have verified that
(\ref{3soros_indoklasa}) is in $o \left( n^{\frac{k}{2}}\right)$,
which implies that $\mathbb{I}_n^{k}$ is asymptotically equal to
(\ref{3soros}).
One can compute
explicitly the integrals not involving the function $\phi$.
Namely, use the identity
\[ \int_C^{t_1-t_0} \frac{(t_1-t_0-x)^l}{\sqrt{x-C}} dx =
(t_1-t_0-C)^{l+1/2} \frac{\Gamma(l+1) \Gamma(1/2)}{\Gamma(l+3/2)}\]
$k-2$ times, to deduce the following formula from (\ref{3soros}):
\begin{eqnarray}
&&\mathbb{I}_n^{k} \sim n^{\frac{k}{2}} c_0^k k! \sigma^{-k}
 (2 \pi)^{-\frac{k-1}{2}} \sqrt{t_1-t_0}
\left[ \phi\left( \frac{b-a}{\sigma \sqrt{t_1-t_0}}\right)\right]^{-1}
\left[ \Gamma \left( \frac{1}{2}\right)\right]^{k-1}
\frac{1}{\Gamma \left( \frac{k-1}{2}\right)} \nonumber \\
&& \iint\limits_{0< s_1<s_2<t_1-t_0} ds_1ds_2
\frac{\phi\left( \frac{a}{\sigma \sqrt{s_1}}\right)}{\sqrt{s_1}}
\frac{\phi\left( \frac{b}{\sigma \sqrt{s_2-s_1}}\right)}{\sqrt{s_2-s_1}}
%\frac{1}{\sqrt{s_1}} \frac{1}{\sqrt{s_2-s_1}}
(t_1-t_0-s_2)^{\frac{k}{2}-\frac{3}{2}} \label{intrepr}
\end{eqnarray}
for $k \geq 2$ (for $k=1$ a simpler formula holds). Finally,
one can slightly simplify the formula (\ref{intrepr}), since
$\Gamma(1/2) = \sqrt \pi$.
In order to complete the method of moments, on the one hand, one needs
to prove that
\begin{equation}
\label{momentlim}
\lim_{n \rightarrow \infty} \mathbb{I}_n^{k} c_0^{-k} n^{-k/2}
= \mathbb{J}^{k},
\end{equation}
where $\mathbb{J}^{k}$ is the $k$-th
moment of $\mathfrak{L}^{a,t_0 \leadsto b, t_1}_{t_1}$.
It is easy to derive from the formulas computed in \cite{Borodin} and
 \cite{Pitman} that
\[ Prob\left( \mathfrak{L}^{a,t_0 \leadsto b, t_1}_{t_1} > y\right) =
\exp \left[ - \frac{1}{2 \sigma^2 (t_1-t_0)}  \left( \left( |a| + |b|
 +\sigma^2 y \right)^2 - (b - a)^2 \right) \right], \]
whence $\mathbb{J}^{k}$
can be expressed with an integral.
On the other hand, one needs to prove that
\begin{equation}
\label{moments}
\limsup_{k \rightarrow \infty} \left(
\frac{\lim_{n \rightarrow \infty} \mathbb{I}_n^{k} c_0^{-k} n^{-k/2}}
{k!}\right)^{1/k} < \infty
\end{equation}
so as to verify that the limit distribution is uniquely determined.
%The above asymptotic representation of $\mathbb{I}_n^{k}$
Observe that (\ref{intrepr}) immediately implies (\ref{moments}),
but proving (\ref{momentlim})
%$\lim_{n \rightarrow \infty} \mathbb{I}_n^{k} c_0^{-k} n^{-k/2}
%= \mathbb{J}^{k}$
turns out to be a nontrivial computation.\\
That is why we need to argue in a slightly different way.
Namely, we are going to prove the first statement of the Lemma
for random walks and then - since the moments for the
random walk have the same asymptotic behavior, and these
moments do converge - we arrive at the original statement.\\
%$\mathfrak{X}, \mathfrak{Y}$
To be more precise, pick a one dimensional
simple symmetric random walk that starts from $a_n$ and
denote its position
after $\lfloor n(t_1-t_0) \rfloor$ steps by $Y_n$. Similarly, its
total number of visits to the origin until
$\lfloor n(t_1-t_0) \rfloor$ is denoted by $Z_n$.
%Further, denote by $(X,Y)$ the position $X$ and the local
%time $Y$ until $t_1-t_0$ at the origin of a standard BM starting from
%$a$.
Write $f_1$ for the probability density function
of $\mathfrak B^{a,t_0}_{t_1}$ with $\sigma$ being replaced by $1$
(that is, $f_1(y)= \frac{1}{\sqrt{t_1-t_0}}
\phi(\frac{y-a}{\sqrt{t_1-t_0}}) $).
Similarly, $F_{2|1}(z|y)$ stands for the conditional
cumulative distribution function of
$\mathfrak L^{a,t_0}_{t_1}$ under the
condition $\mathfrak B^{a,t_0}_{t_1}= y$, again with
$\sigma$ replaced by $1$. Note that $F_{2|1}(z|y)$
is the cumulative distribution function of
$\mathfrak L^{a,t_0 \leadsto y, t_1}_{t_1}$.
Let $y$ be a real number and
$y_n = \lfloor y \sqrt n \rfloor$.\\
The following two statements are well known for random walks
(see for example \cite{Borodin} and \cite{R}):
\begin{equation}
\label{rw1}
Prob \left( Y_n < y_n,
\frac{Z_n}{\sqrt n} < z \right) \rightarrow
Prob \left( \mathfrak B^{a,t_0}_{t_1}< y, \mathfrak L^{a,t_0}_{t_1} <z
\right),
\end{equation}
and
\begin{equation}
\label{rw2}
\frac{\sqrt n}{2} Prob \left(
Y_n \in \{ y_n, y_n+1 \} \right) \rightarrow
f_1(y).
\end{equation}
We want to prove that
\begin{equation}
\label{rw3}
\frac{\sqrt n}{2} Prob \left( Y_n \in \{ y_n, y_n+1 \},
\frac{Z_n}{\sqrt n} < z \right) =: p_n(y,z) \rightarrow
f_1(y)F_{2|1}(z|y).
\end{equation}
Note that in (\ref{rw2}) and (\ref{rw3}) the division by $2$ is needed because
of the periodicity of the random walk (i.e. it can return to the origin
only in even number of steps).
Also notice that using (\ref{rw2}), one easily sees that
(\ref{rw3}) is equivalent to the first statement of the Lemma
for simple symmetric random walks. Further, we mention that
(\ref{rw3}) is proved in \cite{Takacs} for
the case $y=a$. The well known local limit theorem for random walks,
and our previous computation yield that
%In the general case, using (\ref{rw2}), one easily sees that
%(\ref{rw3}) is equivalent to the statement that
%under the condition $X_n \in \{ x_n, x_n+1 \}$,
%$Y_n / \sqrt n$ converges weakly to $Y|X=x$.
the $k$-th moment of $Z_n / \sqrt n$, under the condition
 $Y_n \in \{ y_n, y_n+1 \}$,
have the same asymptotics
as $\mathbb{I}_n^{k} c_0^{-k} n^{-k/2}$, with $b$ replaced by $y$,
and $\sigma =1$.
Thus (\ref{moments}), the method of moments and (\ref{rw2}) imply
that the distribution of $Z_n / \sqrt n$ - under the condition
 $Y_n \in \{ y_n, y_n+1 \}$ - weakly converges to a uniquely determined
 limit distribution. Whence,
$p(y,z):=\lim_{n \rightarrow \infty} p_n(y,z)$ exists.
Now suppose that there exist some $y_0,z_0$ such that
$p(y_0,z_0) \neq f_1(y_0)F_{2|1}(z_0|y_0)$. For this
fixed $z_0$, $f_1(y)F_{2|1}(z_0|y)$
is clearly continuous in $y$. Further,
since the integral representation in (\ref{intrepr})
is continuous in $b$, the method of moments imply that the
limit distribution, as $n \rightarrow \infty$,
of $Z_n / \sqrt n$ - under the condition  $Y_n \in \{ y_n, y_n+1 \}$ -,
continuously depends on $y$ (with respect to the weak topology). Hence,
$p(y,z_0)$ is also continuous in $y$. Thus one can find an interval
$I$ containing $y_0$ such that
$\int_{I} p(y,z_0) dy \neq
\int_{I} f_1(y)F_{2|1}(z_0|y) dy$, which is a
contradiction to (\ref{rw1}). So we have verified
(\ref{rw3}). But (\ref{rw3}) together with
(\ref{moments}) implies (\ref{momentlim}) and the first assertion of the
Lemma.
\end{proof}

Now, we prove of the convergence of one dimensional
distributions by a standard argument.
That is, we need that for any open intervals
$A$, $B$,
\begin{equation}
\label{global1}
{\bf P} \left(
\frac{S^{(per)}_{nt_0}}{\sqrt n} \in A
, \frac{L_{nt_0}}{\sqrt n} \in B \right) \rightarrow
Prob \left(
\mathfrak{B}_{t_0} \in A, c_0
  \mathfrak{L}_{t_0} \in B \right).
\end{equation}
The second statement of Lemma \ref{slemma1} implies
the local version of (\ref{global1})
in the first coordinate, namely
\begin{eqnarray}
%&&\varphi_n(k/ \sqrt n) :=
&&\sqrt n {\bf P} \left(
S^{(per)}_{\lfloor nt_0 \rfloor} = \lfloor x \sqrt n \rfloor
, \frac{L_{nt_0}}{\sqrt n} \in B \right) \label{global2} \\
&=&
\frac{1}{\sigma \sqrt{ t_0}}
\phi \left( \frac{x}{\sigma \sqrt{ t_0}} \right)
Prob \left(
c_0
\mathfrak{L}^{0,0 \leadsto x, t_0}_{t_0} \in B \right)
+ o(1) \nonumber.
  \end{eqnarray}
Now, define the real function $\varphi_n$,
by setting $\varphi_n(x)$ to be equal to (\ref{global2}).
Note that for fix $n$, $\varphi_n$ is constant on the
intervals $[k/\sqrt n,
(k+1)/ \sqrt n)$ for any integer $k$.
%Then the difference of
%$\int_A \varphi_n(x) dx$ and the left hand side of
%(\ref{global2}) is bounded by a constant times $n^{-1/2}$.
We have for any $x$,
\[\varphi_n(x) \rightarrow
\frac{1}{\sigma \sqrt{ t_0}} \phi \left( \frac{x}{\sigma \sqrt{ t_0}} \right)
Prob \left(
c_0
\mathfrak{L}^{0,0 \leadsto x, t_0}_{t_0} \in B \right) =: \varphi (x).
\]
Thus, by Fatou's lemma,
\begin{equation}
\label{Fatou1}
\liminf_n \int_A \varphi_n(x) dx \geq
\int_A \varphi(x) dx = Prob \left(
\mathfrak{B}_{t_0} \in A, c_0
  \mathfrak{L}_{t_0} \in B \right).
\end{equation}
Analogously,
\begin{equation}
\label{Fatou2}
\liminf_n \int_{A^c} \varphi_n(x) dx \geq
\int_{A^c} \varphi(x) dx = Prob \left(
\mathfrak{B}_{t_0} \in A^c, c_0
  \mathfrak{L}_{t_0} \in B \right).
  \end{equation}
As it was already mentioned in the beginning of the proof of
Lemma \ref{slemma1}, the rescaled local times converge to the
appropriate limit. Thus,
\begin{eqnarray}
\label{Fatou3}
&&
\int_{A} \varphi_n(x) dx + \int_{A^c} \varphi_n(x) dx =
 {\bf P} \left( \frac{L_{nt_0}}{\sqrt n} \in B \right) \rightarrow
\nonumber \\
&& Prob \left( c_0
  \mathfrak{L}_{t_0} \in B \right) =
\int_{A} \varphi(x) dx + \int_{A^c} \varphi(x) dx.
\end{eqnarray}
Now, using (\ref{Fatou1}), (\ref{Fatou2}) and
(\ref{Fatou3}), we conclude that
the inequalities in (\ref{Fatou1}) and (\ref{Fatou2}) are,
in fact, equalities and the $\liminf$ can be replaced by $\lim$.
This, together with the observation that
 the difference of
 $\int_A \varphi_n(x) dx$ and the left hand side of
 (\ref{global1}) is bounded by a constant times $n^{-1/2}$, implies
(\ref{global1}). \\
The convergence of any finite dimensional marginals can be
proven analogously, as we proved the one dimensional ones.
The only main difference is that one needs a multiple version
of the statements of
Lemma \ref{slemma1}, but its proof is also analogous.\\

Now we turn to the proof of tightness.
Proposition \ref{WIP} implies that the first coordinate
converges weakly to the desired limit (in $C[t_0,1]$ thus in
$D[t_0,1]$ as well), hence is tight, too. We are going to
establish the tightness of the local times. Then it will follow that
\[ \left( \frac{S^{(per)}_{nt}}{\sqrt n}, \frac{L_{nt}}{\sqrt n} \right)_
{t \in [0,1]}\]
is tight, by definition.\\
Since the process $L_{nt}$ is nondecreasing in $t$,
tightness, in fact, can be deduced from the convergence of finite dimensional
distributions. Namely,
Theorem 15.2 in \cite{Billingsley} yields that we only have to
verify the following two statements:
\begin{enumerate}
\item For each $\eta >0$ there is a $d \in \mathbb{R}$ such that
\begin{equation}
\label{feszesseg1}
{\bf P} \left( \frac{L_{n}}{\sqrt{n}} >  d \right) < \eta, \textbf{   } n \geq 1.
\end{equation}
\item For each positive $\eta$ and $\varepsilon$ there is a $\delta$,
$0 < \delta <1$ and an integer $n_0$ such that
\begin{equation}
\label{feszesseg2}
{\bf P} \left( w_{\frac{L_{nt}}{\sqrt n}}(\delta) \geq
\varepsilon \right) \leq \eta, \textbf{   } n \geq n_0.
\end{equation}
Here,
\[ w_{\psi} (\delta) = \inf_{\{ t_i \}} \max_{0<i \leq r}
( \lim_{\tau \nearrow t_i} \psi(\tau) - \psi(t_{i-1})) ,\]
where the infimum is taken over finite sets $\{ t_i \}$, for
which $0 <t_1 < ... < t_r =1$, $t_i - t_{i-1} > \delta$
for all $i$.

\end{enumerate}
Since we have just verified that $L_n/{\sqrt n}$ converges
weakly, (\ref{feszesseg1}) follows.\\
Again, the convergence of finite dimensional distributions
implies that
for fix $\eta > 0$ and
$\varepsilon > 0$ one can find $\delta > 0$ and $n_0$ such that
for all $n > n_0$, $0 \leq k_1 \leq \lfloor 1/ \delta \rfloor$
\[ {\bf P} \left( \frac{\#
\{ k: n k_1 \delta<k<n (k_1+1)\delta, S^{(per)}_k \in \mathcal{J}\}}
{\sqrt{ n \delta}} > \frac{\varepsilon}{\sqrt{\delta}} \right)
< \eta \delta,\]
Now the equidistant partition $\{ t_i \}$ is enough to verify
(\ref{feszesseg2}). Thus we have finished the proof of Proposition
\ref{allitas1}.

\end{proof}

\section{Proof of Theorem \ref{tetel1}}

Note that, though in its spirit our statement is very close to the results of \cite{D-Sz-V}, their proof cannot be applied here
since the limiting process is not strong Markovian (see Remark
\ref{remark1}) thus leaving no chance to apply the martingale method.
Thus we need to argue in a more direct way, using Proposition \ref{allitas1}. In Subsection
\ref{proof1.1} we prove the first statement of the theorem.
That proof with trivial
modifications is easily applicable to cases 2 and 3.
The only non trivial modification is needed in case 4, which is treated in Subsection \ref{proof1.2}.
Everywhere in this Section, we also use the notations introduced in Section 2.

\subsection{Proof of case 1}
\label{proof1.1}

In order to prove the statement, we need some technical lemmas.\\

\begin{Lemma}
\label{slemma2}
Let $E$ and $F$ be any Polish spaces, $X, X_n$ any random variables
taking values in the space $E$ such that $X_n \Rightarrow X$.
Then for any continuous function $f:E \rightarrow F$ one has
$(X_n, f(X_n)) \Rightarrow (X, f(X))$ in the product topology.
\begin{proof}
Pick any $U \subset E \times F$ open set and define
$V = \{ x \in E : (x, f(x)) \in U   \}$.
If $x \in V$, then one can find an open product set
$U_x = E_x \times F_x \subset U$  containing $(x,f(x))$.
Since
$f^{-1}(F_x)$ is open,
$x \in E_x \cap f^{-1}(F_x)$ is also open. Now
$x \in E_x \cap f^{-1}(F_x) \subset V$ implies that
V is open, too. Thus
\[Prob ((X_n,f(X_n)) \in U) = Prob (X_n \in V)\]
and the Portmanteau Theorem (see \cite{Billingsley}, for instance)
yield the statement.
\end{proof}
\end{Lemma}

Next, we prove the following Lemma which is an extension
of Proposition \ref{allitas1}.

\begin{Lemma}
\label{allitas2}
\[\left( \frac{S^{(per)}_{nt}}{\sqrt n}, \frac{L_{nt}}{\sqrt n},
\int_{\tau = t_0}^{t} \frac{1}{\sqrt \tau} d \frac{L_{n \tau}}{\sqrt n}
\right)_
{t \in [t_0,1]} \Rightarrow
\left( \mathfrak{B}_t, c_0
\mathfrak{L}_t ,
c_0 \int_{\tau = t_0}^{t} \frac{1}{\sqrt \tau} d \mathfrak{L}_{\tau}
\right)_{t \in [t_0,1]},\]
%\left( \mathcal{W}_t,c_0 \mathcal{L}_t,
%c_0 \int_{\tau = t_0}^{t} \frac{1}{\sqrt \tau} d \mathcal{L}_{\tau} \right)_
%{t \in [t_0,1]},\]
where the left hand side is understood as a random variable with respect
to the probability measure ${\bf P}$ and
$\Rightarrow$ stands for
weak convergence in the Skorokhod space $D_{\mathbb{R}^3} [t_0,1]$.
\end{Lemma}

\begin{proof}
Use Proposition \ref{allitas1} and Lemma \ref{slemma2} with the choice
\begin{eqnarray*}
E &=& \{ \psi=(\psi_1, \psi_2) \in D_{\mathbb{R}^2}[t_0,1]: \psi_2
\text{ is non decreasing}\}, \\
F &=& D[t_0,1] \\
f((\psi_1,\psi_2)) &=& \left( \int_{\tau=t_0}^t \frac{1}{\sqrt \tau}
d\psi_2 \right)_{t \in [t_0,1]}
\end{eqnarray*}
to infer Lemma \ref{allitas2}.
\end{proof}

Note that we needed to restrict the processes of Proposition \ref{allitas1}
to $D_{\mathbb{R}^2} [t_0,1]$ in order to the above stochastic integrals be finite.
(This technical difficulty can be avoided in the proof of case 2.)
Finally, we will need Le Cam's famous inequality which was proven in \cite{LC}.
\begin{Lemma}
\label{lemmaPoi}
 Assume $\Sigma_m$ is the sum of $m$ independent, non-identically distributed
 Bernoulli random variables
 $\varepsilon_j; \ 1 \le j \le m$ such that $Prob(\varepsilon_j = 1) = p_j$. Then
 \[
 \sum_{k=0}^{\infty} \Big| Prob (\Sigma_m = k) - e^{-\lambda}\lambda^k/k! \Big| \le
 2 \sum_{j=1}^m p_j^2,
 \]
where $\lambda = p_1 + \dots + p_m$.
\end{Lemma}

Now, we can proceed to the proof of case 1 of Theorem \ref{tetel1}.
First, we are going to prove a simplified version of the assertion,
namely, the convergence of
the measures $\mu^{\searrow}_n$ restricted to $C[t_0,1]$.
Then the statement of the
first part of the Theorem will follow easily.\\
Note that one can think about our model as having two sources of randomness.
The first one is the choice of $x$ and the second is the choice of $\xi$'s.
In Section 2, we were only dealing with the first source, but now we are going
to treat the second one, as well.\\
It would be more convenient to consider  $S^{\searrow}_n$ as if
the time instants of the reflections on the wall $W_k$
($1 \leq k\leq n$)
were not computed.
Since Proposition \ref{allitas1}
and the sctterer configuration being symmetric to the $y$-axis
imply that
$ |\{i \leq n: \mathcal{F}^{\searrow}_i
 \dots \mathcal{F}^{\searrow}_1 \in \mathcal{J} \}|$
  is asymptotically of order $\sqrt n$, the diffusively scaled
   limits of $S^{\searrow}_n$ and of this "modified $S^{\searrow}_n$"
    (i.e. when we do not count the reflections on the wall)
    have the same limit. Thus it is sufficient to prove our statement
    for the "modified $S^{\searrow}_n$" - which will also be denoted by
    $S^{\searrow}_n$ in the sequel. \\
    Note that the assumption of the periodic scatterer configuration
    being symmetric implies
    \begin{equation}
    \label{modulus}
    | S^{\searrow}_n| = |S^{(per)}_n|.
    \end{equation}

Now for fix $x$, define $p(nt)$ as the probability,
generated by the choice of $\xi_n$'s, of the event that
$S^{\searrow}_{\lfloor nt\rfloor} (x)
S^{\searrow}_{\lfloor nt\rfloor +1} (x) <0$, i.e. after step number
$\lfloor nt\rfloor$, the particle crosses the hole.
Lemma \ref{allitas2} implies that - by Skorokhod's representation theorem, cf. \cite{Billingsley} - there exists a probability
space $(\Omega, \mathbb{Q})$ together with random variables
$(\tilde{X}_n, \tilde{Y}_n, \tilde{Z}_n)$ having the same joint
distribution as
\[\left( \left(\frac{S^{(per)}_{nt}}{\sqrt n}\right)_{t \in [t_0,1]},
 \left( \frac{c}{c_1} \int_{\tau = t_0}^{t} \frac{1}{\sqrt \tau}
d \frac{L_{n \tau}}{\sqrt n}\right)_{t \in [t_0,1]},
\left( \int_{\tau = t_0}^{t}
 p(n \tau)
 d L_{n \tau}\right)_{t \in [t_0,1]} \right)\]
with respect to ${\bf P}$,  and also with random variables
$(\tilde{X}, \tilde{Y})$ having the same joint
distribution as
\[\left( \left( \mathfrak{B}_t\right)_{t \in [t_0,1]},
\left(\frac{cc_0}{c_1} \int_{\tau = t_0}^{t} \frac{1}{\sqrt \tau} d \mathfrak{L}_{\tau} \right)_
{t \in [t_0,1]} \right), \]
such that
$(\tilde{X}_n, \tilde{Y}_n) \rightarrow (\tilde{X} ,\tilde{Y})$
$\mathbb{Q}$-almost surely. Here, $c = \lim_n \alpha_n \sqrt n$
Now, for $\mathbb{Q}$-almost all $\omega \in \Omega$ we define the
measures $\nu(\omega)$, $\nu_n(\omega), \lambda_n(\omega)$ on
$C[t_0,1]$ in the following way. Consider the modulus of
$\tilde{X}(\omega)$, i.e. $|\tilde{X}(\omega)| \in C[t_0,1]$
(if $\tilde{X}(t_0)(\omega)>0$; otherwise consider $-|\tilde{X}(\omega)|$),
pick a Poisson point process - on some abstract probability space
$(\Omega_{\omega}, \mathbb{Q}_{\omega})$ -
with intensity measure $d\tilde{Y}(\omega)$, and denote its point
by $P_1<P_2<... $ N.b. there are finitely many points. If it has $m$
points, put $P_{m+1} =1$. Now
reflect the
subgraph of $|\tilde{X}(\omega)|[P_{2i+1}, P_{2i+2}]$ to the origin
for each $i$
(if $\tilde{X}(t_0)(\omega)>0$; otherwise reflect $-|\tilde{X}(\omega)|$). The
distribution of the resulting random function - with respect to
$\mathbb{Q}_{\omega}$ - generates a measure on $C[t_0,1]$ which we denote by
$\nu(\omega)$. The construction of $\nu_n(\omega)$ is
similar, with two differences.
The first is that one should replace $\tilde{X}$ and
$\tilde{Y}$ by $\tilde{X}_n$ and $\tilde{Y}_n$ and the second is that
instead of the
Poisson point process, one introduces independent Bernoulli
random variables for each discontinuity of the function
$\tilde{Y}_n(\omega)$ with parameters being equal to the size of jump
of $\tilde{Y}_n(\omega)$ at the corresponding discontinuity.
Then denote by $P_1 <P_2 < \dots$ the positions, where the Bernoulli
random variables are equal to 1.
Finally, $\lambda_n(\omega)$
is defined the same way as $\nu_n(\omega)$ with $\tilde{Y}_n$ being
replaced by $\tilde{Z}_n$.\\
Using Lemma \ref{lemmaPoi}, one can infer that for $\mathbb{Q}$-almost
all $\omega$, $\nu_n(\omega) \Rightarrow
\nu(\omega)$ on $C[t_0,1]$.
Further, $\alpha_n \sqrt n
\rightarrow c$ implies that for any fixed $x$ and $\varepsilon >0$, if
$L_{\lfloor n \tau \rfloor -1} <L_{\lfloor n \tau \rfloor}$, i.e.
in $\lfloor n \tau \rfloor$ steps the particle arrives to $\mathcal J$, then
$|p(\lfloor n \tau \rfloor) -
c/(c_1 \sqrt{\lfloor n \tau \rfloor})| < \varepsilon / \sqrt n$
assuming that $n$ is large enough.
Whence, one can naturally couple the
Bernoulli
distributed random variables used by the definition of $\nu_n$ and
$\lambda_n$ in such a way that the resulting random
functions in $C[t_0,1]$ coincide on a subset of $\Omega_\omega$, whose
$\mathbb{Q}_\omega$ measure tends to $1$ as $n \rightarrow \infty$. Consequently,
$\lambda_n(\omega) \Rightarrow
\nu(\omega)$ on $C[t_0,1]$ for $\mathbb{Q}$-almost
all $\omega$, too.\\
Define the measures $\varrho$ and $\varrho_n$ on $C[t_0,1]$ by
\begin{eqnarray*}
\varrho(A) &=& \int_{\Omega} \nu(\omega)(A) d\mathbb{Q}(\omega),\\
\varrho_n(A) &=& \int_{\Omega} \lambda_n(\omega)(A) d\mathbb{Q}(\omega).
\end{eqnarray*}
Using that $\lambda_n(\omega) \Rightarrow
\nu(\omega)$ for $\mathbb{Q}$-almost
all $\omega$, Fatou's lemma and the Portmanteau theorem, we obtain for any
$A \subset C[t_0,1]$ open set:
\begin{eqnarray*}
\liminf_n \varrho_n(A) &=& \liminf_n \int_{\Omega} \lambda_n(\omega)(A) d\mathbb{Q}(\omega) \\
&\geq& \int_{\Omega} \liminf_n \lambda_n(\omega)(A) d\mathbb{Q}(\omega)
\geq \int_{\Omega} \nu(\omega)(A) d\mathbb{Q}(\omega) = \varrho(A).
\end{eqnarray*}
Whence - by the Portmanteau theorem, again - $\varrho_n \Rightarrow \varrho$ on $C[t_0,1]$.\\
Observe that by construction,
%$\varrho_n$ is the restriction
%of the measure $\mu^{\searrow}_n$ to $C[t_0,1]$, and similarly,
$\varrho$ is the measure on $C[t_0,1]$ generated
by a QRBM($c c_0/c_1 ,\sigma$).
Similarly, $\varrho_n$ is the restriction of
${\bf W}^{\searrow}_n$ to $C[t_0,1]$.\\
%Thus, we obtain the statement of the first part of Theorem \ref{tetel1} for
%the restriction of the appropriate measures to $C[t_0,1]$
%(in fact, with the constant
%$c_2=c_0/c_1$).\\
%Note that by construction, $\varrho_n$ is the same measure as
%$\mu^{\searrow}_n$ restricted to $C[t_0,1]$ (with the modification of
%$S_n^{\searrow}$ introduced in the beginning of Step 3). Thus
%This
%implies the convergence of
%the measures $\mu^{\searrow}_n$ restricted to $C[t_0,1]$
%to the desired limit (in fact, with the constant
%$c_2=c_0/\sum_{k=1}^K \mathcal{J}_{k,r}-\mathcal{J}_{k,l}$).
Now, one can easily prove the first part of the Theorem.
Since the choice of $t_0$ was arbitrary, a limit theorem of any
finite dimensional distributions is implied by the above computation.
The tightness is also easy since the moduli of the
random functions are tight.
Thus we have finished the proof of the first part of the Theorem
(in fact, with the constant
$c_2=c_0/c_1$).\\

\subsection{Proof of case 4}
\label{proof1.2}

As in the previous subsection, the tightness is trivial since the moduli of the
random functions are tight. The convergence of one dimensional
distributions follows from symmetry and Proposition \ref{LLT}.
Here, we are only going to prove the convergence of two dimensional marginals
since
the convergence of any finite dimensional ones can be proven similarly.\\
%Again, for the two dimensional marginals, we prove the local version of
%the limit theorem which implies the global version.
The idea of the proof is that we know the convergence of
$(| S^{\searrow}_{\lfloor n t_0 \rfloor}| / \sqrt n,
|S^{\searrow}_{\lfloor n t_1 \rfloor}|/ \sqrt n)$ to the desired limit,
thus we only need to care about the sign. For the latter, assume that
$S^{\searrow}_{\lfloor n t_0 \rfloor}/ \sqrt n$ is
in a fixed positive interval,
while $|S^{\searrow}_{\lfloor n t_1 \rfloor}| / \sqrt n$
is in another fixed positive interval. Using Proposition \ref{allitas1},
and the results of the previous subsection,
we can estimate the
asymptotic probability of the local time being zero under the above condition.
If the local time is zero, then the trajectory avoids the origin, hence
$S^{\searrow}_{\lfloor n t_1 \rfloor}>0$. If not, then the particle
arrives at the origin eventually in $[n t_0, n t_1]$, and we need to verify
that at time $n t_1$, it will end up in the positive half line with probability
$1/2$. The heuristic reason for this is that once it is near the origin,
since $\alpha_n \sqrt n$ is large, it will cross the holes many times,
and thus forget that it came from the positive half-line.
This argument will imply that the weak limit must be the two dimensional
marginal of the BM.\\
Let us make the above argument precise. To do so, we
will use the notations of the previous subsection.
Especially, introduce the modification of $S^{\searrow}_n$ as in the
previous subsection. Thus (\ref{modulus}) still holds.
Fix $0<t_0<t_1 \leq 1$ and $J_0$, $J_1$ compact subintervals of
$\mathbb{R}_+ \cup \{ 0\}$. Our aim is to prove that
\begin{equation}
\label{case4new-1}
{\bf P} \left(
\frac{ S_{\lfloor nt_0 \rfloor}^{\searrow}}{\sqrt n} \in J_0,
\frac{ S_{\lfloor nt_1 \rfloor}^{\searrow}}{\sqrt n} \in J_1
\right) \rightarrow Prob \left(
\mathfrak{B}_{t_0} \in J_0,\mathfrak{B}_{t_1} \in J_1 \right),
\end{equation}
and
\begin{equation}
\label{case4new0}
{\bf P} \left(
\frac{ S_{\lfloor nt_0 \rfloor}^{\searrow}}{\sqrt n} \in J_0,
\frac{ S_{\lfloor nt_1 \lfloor}^{\searrow}}{\sqrt n} \in - J_1
\right) \rightarrow Prob \left(
\mathfrak{B}_{t_0} \in J_0,\mathfrak{B}_{t_1} \in - J_1 \right),
\end{equation}
as $n \rightarrow \infty$. Once we verify (\ref{case4new-1})
and (\ref{case4new0}), by symmetry, they will also hold true for
$J_0$ being a compact interval in $\mathbb{R}_- \cup \{ 0\}$, and
hence the convergence of two dimensional marginals will follow.\\
Define the probabilities
%\[p_{a,b} = Prob \left( \forall s: t_0 <s <t_1: \mathfrak{B}^{a,t_0 }_{s} >0
%| |\mathfrak{B}^{a,t_0}_{t_1}| = b \right),\]
%and analogously,
\[p_{J_0,J_1} = Prob \left( \forall s: t_0 <s <t_1: \mathfrak{B}_{s} >0
| \mathfrak{B}_{t_0} \in J_0,|\mathfrak{B}_{t_1}| \in J_1 \right).\]
Now let $A$ be the set of functions $\psi$ in $C[0,1]$ for which
$\psi(t_0) \in J_0$, $\forall s: t_0 <s <t_1:
\psi(s) >0$, and $\psi(t_1) \in J_1$. Then the Wiener measure of
$\partial A$ is zero, thus
Proposition \ref{WIP} implies that
%For every $\varepsilon >0$, there is a $N$
%such that for $n>N$, we have
\begin{equation}
\label{case4new1}
{\bf P} \left( \forall t_0 < s < t_1: S_{ns}^{(per)} >0 |
\frac{S_{\lfloor nt_0 \rfloor}^{(per)}}{\sqrt n} \in J_0,
\frac{| S_{\lfloor nt_1 \rfloor}^{(per)}|}{\sqrt n} \in J_1
\right) \rightarrow  p_{J_0,J_1},
%\nonumber \\
%&\in& [\min_{a \in J_1, b_in J_2}{p_{a,b}} - \varepsilon,
%\max_{a \in J_1, b \in J_2}{p_{a,b}} + \varepsilon] \label{case4new1}
\end{equation}
as $n \rightarrow \infty$. On the other hand, the strong Markov
property of the BM obviously implies
\begin{equation}
\label{case4new1.5}
 Prob \left( \mathfrak{B}_{t_1} \in J_1 |
 \mathfrak{B}_{t_0} \in J_0, \exists s: t_0 <s <t_1: \mathfrak{B}_{s} =0,
  |\mathfrak{B}_{t_1}| \in J_1 \right) = \frac12.
\end{equation}
Now, with the notation
\[ \mathcal{A}_4(n)=
\{x:
\frac{S_{\lfloor nt_0 \rfloor}^{\searrow}(x)}{\sqrt n} \in J_0,
\exists t_0 < s < t_1: S_{ns}^{(per)}(x) =0,
\frac{| S_{\lfloor nt_1 \rfloor}^{\searrow}(x)|}{\sqrt n} \in J_1
\},\]
we want to prove that
\begin{equation}
\label{case4new2}
{\bf P} \left(  \frac{ S_{\lfloor nt_1 \rfloor}^{\searrow}}{\sqrt n} \in J_1 |
\mathcal{A}_4(n)
\right) \rightarrow \frac12.
%\nonumber \\
%&\in& [\min_{a \in J_1, b_in J_2}{p_{a,b}} - \varepsilon,
%\max_{a \in J_1, b \in J_2}{p_{a,b}} + \varepsilon] \label{case4new1}
\end{equation}
Note that combining Proposition \ref{WIP},
(\ref{modulus}), (\ref{case4new1}), (\ref{case4new1.5}) and (\ref{case4new2}),
one can deduce
(\ref{case4new-1}) and (\ref{case4new0}).
Thus it remains to prove (\ref{case4new2}).\\
To do so, first
observe that by
Proposition \ref{allitas1}, for every
$\varepsilon >0$ there exists $\delta >0$ and $N$ such that
for all $n >N$,
\begin{equation}
\label{case41}
{\bf P} \left( L_{[ nt_0, nt_1 ]} > \delta \sqrt n |
\mathcal{A}_4(n)
\right) > 1- \varepsilon.
\end{equation}
Now consider the Markov transition matrices
 \[ A_p = \left( \begin{array}{ccc}
  1-p & p \\
   p & 1-p
    \end{array} \right)
     \]
      on the space $\{ +,-\}$
       and a time dependent Markov chain $M_k$ such that
        $M_0 = +$ and the transition between $k$ and $k+1$ is
	 described by $A_{p_k}$ with some numbers $p_k$.
Now, for a fixed $x \in \mathcal{M}^{(per)}$ and $n$, define
$D_k(x)$ as the $k$-th leftmost discontinuity of the function
$s \rightarrow L_{\lfloor ns \rfloor}(x)$ on $s \in [t_0,
t_1)$. With the choice $p_k(x) = {\alpha_{n D_k(x)}}/
{c_1}$,
$1 \leq k
\leq L_{[\lfloor nt_0 \rfloor, \lfloor n t_1 \rfloor -1] }(x)$
for each $x$, one easily sees
that for $n$ large enough, the probability
in (\ref{case4new2}) is equal to
\begin{equation}
\label{case43}
\frac{1}{{\bf P}(\mathcal{A}_4(n)) }
\int_{\mathcal{A}_4(n)}
Prob(M_{L_{[\lfloor nt_0 \rfloor, \lfloor n t_1 \rfloor -1] }(x)}=+)
d {\bf P}(x).
\end{equation}
	  If fact, this is only true for the case of $\mu_n^{\searrow}$,
	  while in the case of $\mu_n^{\equiv}$, one needs to set
	  $p_k(x) = {\alpha_{n}}/
	  {c_1}$. \\
	  On the other hand, elementary computations show that if one
	  selects sequences $B(n) \rightarrow \infty, m(n) \rightarrow \infty$
	  and non-negative numbers $p_{k,n}, 1 \leq n, 1 \leq k \leq m(n)$, then
	  with the
	  transition matrices corresponding to $p_{1,n}, ..., p_{m(n),n}$,
\begin{eqnarray*}
&&Prob(M_{m(n)}=+) \\
&=&
\left( \begin{array}{ccc}
1 & 0
\end{array} \right)
A_{p_{1,n}} \dots A_{p_{m(n),n}}
\left( \begin{array}{ccc}
1 \\
0
\end{array} \right) \\
&=&
 \left( \begin{array}{ccc}
  1/\sqrt{2} & 1/\sqrt{2}
   \end{array} \right)
    \left( \begin{array}{ccc}
     1 & 0 \\
      0 & \prod_{k=1}^{m(n)} (1-2p_{k,n})
       \end{array} \right)
        \left( \begin{array}{ccc}
	 1/\sqrt{2} \\
	  1/\sqrt{2}
	   \end{array} \right) \\
	   &=& 1/2 + o(1),
	   \end{eqnarray*}
	   as $n \rightarrow \infty$.
	   Further, $o(1)$ converges to zero uniformly in $m$ and $p_k$
	   if $m(n) > \delta \sqrt n$ and
	   $\min_{1 \leq k \leq m(n)} \{ p_{k,n} \sqrt n \} > B(n)$.
	   Now, choose
	   $B(n)=\min_{N \geq \lfloor n t_0 \rfloor}
	   \{ \alpha_{N} \sqrt N /c_1\}$.
	   This estimation together with
	   (\ref{case41}) and (\ref{case43}) yields (\ref{case4new2}).\\

{\bf ACKNOWLEDGEMENTS.}
The support of the Hungarian National Foundation for Scientific
Research grant No. K 71693 is gratefully acknowledged.
The authors thank the kind hospitality of the Fields
Institute (Toronto) where - during June 2011 -  part of this work was done.
They are also thankful to the referee for his constructive remarks leading to the improvement of the exposition.

\end{document}